\documentclass[reqno,11pt]{amsart}
\usepackage{a4wide,color,eucal,enumerate,mathrsfs}
\usepackage[normalem]{ulem}
\usepackage{amsmath,amssymb,epsfig,amsthm} 
\usepackage[latin1]{inputenc}
\usepackage{psfrag}
\usepackage{pdfpages}

\def\az{\alpha}

\def\dist{{\mathop\mathrm{\,dist\,}}}
\def\loc{{\mathop\mathrm{\,loc\,}}}

\def\ez{\epsilon}

\def\gz{{\gamma}}

\def\bbint{{\ifinner\rlap{\bf\kern.35em--}
\hspace{0.078cm}\int\else\rlap{\bf\kern.45em--}\int\fi}\ignorespaces}

\def\diam{{\mathop\mathrm{\,diam\,}}}

\newtheorem{thm}{Theorem}[section]
\newtheorem{lem}[thm]{Lemma}
\newtheorem{prop}[thm]{Proposition}
\newtheorem{cor}[thm]{Corollary}
\newtheorem{defn}[thm]{Definition}
\numberwithin{equation}{section}

\theoremstyle{remark}
\newtheorem{rem}[thm]{Remark}

\def\bint{{\ifinner\rlap{\bf\kern.35em--}
\int\else\rlap{\bf\kern.45em--}\int\fi}\ignorespaces}

\usepackage{graphicx}
\usepackage{import}
\usepackage{xifthen}
\usepackage{pdfpages}
\usepackage{transparent}

\title[A growth estimate for  planar Mumford--Shah minimizers]{A growth estimate for the   planar Mumford--Shah minimizers at a tip point: An alternative proof of David--L\'eger}
\author{Yi Ru-Ya Zhang}
\date{\today}

\address{Academy of Mathematics and Systems Science, the Chinese Academy of Sciences, Beijing 100190, China}
\email{yzhang@amss.ac.cn}

\thanks{The  author is funded by National Key R\&D Program of China (Grant No. 2021YFA1003100), NSFC grant No. 12288201 and  the Chinese Academy of Sciences.}

\subjclass[2020]{49Q20, 49N60}
\keywords{Mumford--Shah problem, John domain, tip points}
\begin{document}

\begin{abstract}
Let $\Omega\subset \mathbb R^2$ be a bounded domain and $u\in SBV(\Omega)$ be a local minimizer of the Mumford--Shah problem in the plane, with $0\in \overline{S}_u$ being a tip point and $B_1\subset \Omega$. Then  there exist  absolute constants $C>0$ and $0<r_0<1$ such that  
$$|u(x)-u(0)|\le C r^{\frac 1 2} \quad \text{ for any } \ x\in B_r \ \text{ and } \ 0<r<r_0. $$
This estimate is a local version of the original one in \cite[Proposition 10.17]{DL2002}. 

Our result is based on a dichotomy  and the John structure  of $\Omega\setminus \overline{S}_u$,   different from the one by David--L\'eger \cite{DL2002} or Bonnet--David \cite[Lemma 21.3]{BD2001}. 
\end{abstract}


\maketitle
\section{Introduction}
The Mumford--Shah functional, introduced by Mumford and Shah in \cite{MS1989}, is a well--known model in image processing. In their seminal paper \cite{DCL1989}, De Giorgi, Carriero, and Leaci established the existence of minimizers for a weaker formulation of the Mumford--Shah problem through direct methods, drawing on a lower semicontinuity result by De Giorgi and Ambrosio \cite{DA1988}.

To be more specific, for any bounded domain $\Omega\subset \mathbb R^n$, it was introduced in \cite{DA1988}  a subspace of  $BV(\Omega)$, denoted by  $SBV(\Omega)$, in which the functions only  has jump discontinuities (see Section 2 for more details). Then for a function $u\in SBV( \Omega)$, its $\lambda$-Mumford--Shah energy on an open set $\Omega \subseteq \mathbb{R}^n$ is defined by
$$
MS_\lambda(u,\,\Omega):= \int_{\Omega}{\left| D u \right|}^2\, dx+ \lambda  \mathcal{H} ^{n-1}\left( S_u \cap \Omega\right),
$$
where $\lambda>0$ and $S_u\subset \Omega$ is the set of discontinuity points of $u$. 
A function $u\in SBV_\loc (\Omega)$ is a local $\lambda$-minimizer  if for any $x\in \Omega$ with $B_r(x)\subset \Omega,\,r>0$ and every open set $U\subset\subset \Omega\cap B_r(x)$, we have $MS_\lambda(u,\,U)<\infty$ together with
$$MS_\lambda(u,\,U)\le MS_\lambda(v,\,U),$$
whenever $\{u\neq v\}\subset\subset U. $ 

The Euler--Lagrange equation \cite[Theorem 7.35]{AFP2000} for  a local minimizer is that, for any $\eta\in C_0^1(\Omega;\, \mathbb R^n),$
\begin{equation}\label{EL}
\int_{\Omega\setminus\overline{S}_u} |Du|^2 {\rm div} \eta - 2 \langle \nabla u, \nabla u\cdot \nabla\eta\rangle \, dx +\lambda \int_{S_u} {\rm div}_\tau \eta \, d\mathcal H^{n-1}=0, 
\end{equation}
where ${\rm div}_\tau$ denotes the tangential divergence. 
This in particular implies that $u$ is harmonic in $\Omega\setminus \overline{S}_u$ and satisfies the zero Neumann boundary value condition on both sides of $\overline{S}_u$; see \cite[(7.42)]{AFP2000}. Moreover the (weak) mean curvature of $\overline{S}_u$ equals to the jump of the gradient $[|Du|^2]^{\pm}$, according to  \cite[Theorem 7.38]{AFP2000}.

Ambrosio, Fusco and Pallara proved in \cite{AFP19971, AFP19972} that, when $\Omega\subset \mathbb R^n$ and $u\in SBV(\Omega)$ is a local Mumford--Shah minimizer, there exists a subset $\Sigma\subset \overline S_{u}$, which is relatively closed and $\mathcal H^{n-1}(\Sigma)=0$, such that
the set $\overline S_{u}\setminus \Sigma$ is the union of $C^{1,\,\frac 1 4}$-hypersurfaces. Moreover, both $u$ and $Du$ have a H\"older continuous extension to $\overline S_{u}\setminus \Sigma$. 
For more details, see also \cite[Theorem 8.1]{AFP2000}, along with the survey \cite{F2003}, and \cite{AFH2003, DFR2014, L2011} for more recent results.

Nevertheless, much more is known about the planar case. In particular, in an earlier result \cite{B1996}, Bonnet proved that an isolated connected component of $\overline{S}_u$ is a finite union of $C^{1,\,1}$-arcs. His result is based on the so--called Bonnet's monotonicity formula, which is applicable when the discontinuity set in the plane is connected. Later, David \cite{D1996} demonstrated a version of $\ez$-regularity for the minimizers, and many additional results regarding Mumford--Shah local minimizers in the plane were established by Bonnet and David in \cite{BD2001}.

In particular, in the monograph \cite[Theorem 69.29]{D2006}, David proved that if 
$\overline{S}_u$ at a tip point $x\in \overline{S}_u\setminus  {S}_u$ with $B(x,\,2)=:B_2(x)\subset \Omega$,   is sufficiently close to a single radius in the Hausdorff distance within   (which implies $\overline{S}_u$ is connected in $B_1(x)$ \cite[Lemma 69.8]{D2006}), then $\overline{S}_u$ is locally $C^{1,\,1}_\loc$ in $B_1(x)$. Recently, this result was independently improved by \cite{AM2019} and \cite{DFG2021}, showing that one actually obtains $C^{1,\,1}$ (and even $C^{2,\,\az}$)-regularity up to the end point of $\overline{S}_u$, again under the assumption that $\overline{S}_u\cap B_2(x)$ is sufficiently close to a single radius in the Hausdorff distance.

The planar results mentioned above generally rely on the a priori assumption that $\overline {S}_u$ is connected. In this paper, we examine the growth of $u$ near a tip point \emph{without} assuming the connectedness of the discontinuity set. This estimate serves as a localized version of the original estimate for  global minimizers presented in \cite[Proposition 10.17]{DL2002}, via a completely different argument; readers may also refer to  \cite[Proposition 4.6.1]{DF2023}.

\begin{thm}\label{main thm}
Let $\Omega\subset \mathbb R^2$ be a bounded domain, $u\in SBV(\Omega)$ be a local minimizer for the Mumford--Shah problem, and $0\in \Omega$ be a tip point, i.e. $0\in \overline S_{u}\setminus S_u$. Suppose that $B_1\subset \Omega$. Then there exist absolute constants $C>0$ and $0<r_0<1$ so that, for any $0<r<r_0$, one has
$$|u(x)-u(0)|\le C r^{\frac 1 2} \quad \text{ for any } \ x\in B_r. $$

\end{thm}

Our estimate is not straightforward because $u\in SBV(\Omega)$ might not satisfy a Poincar\'e-type inequality, even if we can estimate the growth  of the $L^2$-norm of $Du$ within every disk (see Lemma~\ref{A reg}). The discontinuity of $u$ poses a significant challenge. 

It is worth noting that the harmonic conjugate $v$ of $u$ can be defined with H\"older continuity  exponent $\frac 1 2$. This modulus of continuity  follows directly from Lemma~\ref{A reg} as $v\in W^{1,\,2}(\Omega);$ see e.g. \cite[Proposition 4.5.1]{DF2023}. 

\begin{rem}
Some readers might compare our proof with the proof of Lemma 21.3 in \cite{BD2001}. However, it is important to note that they are not the same. 

Firstly, our estimate is directly based on an Ahlfors regularity lemma (see Lemma~\ref{A reg} below), which provides an $L^2$-estimate of $|Du|$. In contrast, the proof of Lemma 21.3 in \cite{BD2001} relies on an  $L^\infty$-estimate, as seen in equation \cite[(21.1)]{BD2001}. 

Secondly, the proof of Bonnet--David assumes that $x$ and $0$ are in the same component of $B\setminus \overline{S}_u$ (see \cite[(21.2)]{BD2001}). However, this assumption is addressed in our work through Corollary~\ref{limit tip} below. 

More precisely, our Lemma~\ref{growth of u} and Lemma~\ref{local john decompose} may share some similarities with the proof of \cite[Lemma 21.3]{BD2001}. To some extent, Lemma~\ref{growth of u} can be seen as an  $L^2$-version of \cite[(21.14), Lemma 21.3]{BD2001}. Similarly, the finiteness of $N$ in Lemma~\ref{local john decompose} is related to the finiteness of the set $E$ in the proof of \cite[Lemma 21.3]{BD2001}. To handle the  $L^2$-integral directly, one needs to use carrot John subdomains (or Boman chains). Therefore,  we  must demonstrate that one can cover $B_r\setminus \overline{S}_u$ with a uniformly finite number of carrot John subdomains whose closures contain $0$, provided that $r$ is smaller than some uniform constant $r_0>0$. 
This part serves as an application of a recent work in \cite[Theorem 1.7]{SZ2024}; it was somehow implicitly used in  \cite[Lemma 21.3]{BD2001}. 
\end{rem}

The proof of the theorem relies on the dichotomy in Proposition~\ref{dichotomy}, which offers a criterion for distinguishing between tip points and jump points. We will defer the proof of this proposition to Section~\ref{Preliminaries}.

\begin{prop}\label{dichotomy}
Let $u\in SBV(\Omega)$ be a local minimizer in the plane, which is not locally constant, and $x\in K:= \overline S_{u}$. 
We define for $r>0$ with $B_r\subset \Omega$, 
$$\Phi_u(x,\,r):=   \frac{r^{-1} \int_{B_r} |Du|^2\, dx}{\inf_{c\in \mathbb R} \bbint_{B_r}|u-c|^2 \,dx}.$$
Then 
  $x$ is a jump point if and only if
\begin{equation}\label{poincare}
\Phi_u(x):=\Phi_u(x,\,0^+):=  \limsup_{r\to 0^+}  \frac{r^{-1}\int_{B_r} |Du|^2\, dx}{\inf_{c\in \mathbb R} \bbint_{  B_r}|u-c|^2 \,dx} \quad \text{ is finite}.
\end{equation}

In other words, 
$x$ is a tip point if and only if 
$
\Phi_u(x,\,0^+)=+\infty.
$
\end{prop}

Note that  $\Phi_u(\cdot)$ takes values in $\mathbb R^+\cup\{+\infty\}$, and
$\Phi_u(x,\,r)$ is continuous with repsect to $x$ for each fixed $r>0$. In addition,
$$\Phi_u(x,\,0^+)=\lim_{s\to 0} \sup_{0<r<s} \Phi_{u}(x,\,r)$$
  is the limit of a monotone non-increasing sequence as $r\to 0$. Thus  
$\Phi_u(\cdot)$ is upper semicontinuous. Thus we conclude the following corollary from Proposition~\ref{dichotomy}.
\begin{cor}\label{limit tip}
The limit of a sequence of tip points must be a tip point. Especially, when $0$ is a tip point of a local minimizer $u\in SBV(\Omega)$ and $B_1\subset \Omega$, the compactness of Mumford--Shah problem yields that, there exists $r_2>0$ such that 
$$\dist(0,\,V)\ge r_2$$
where $V$ is a connected component of $B_1\setminus  \overline S_{u}$ for which $0\notin \partial V$. See Figure~\ref{fig:gamma}. 
\end{cor}

	\begin{figure}[ht]
		\centering
		\def\svgwidth{\columnwidth}
		\resizebox{0.7\textwidth}{!}{
\begingroup%
  \makeatletter%
  \providecommand\color[2][]{%
    \errmessage{(Inkscape) Color is used for the text in Inkscape, but the package 'color.sty' is not loaded}%
    \renewcommand\color[2][]{}%
  }%
  \providecommand\transparent[1]{%
    \errmessage{(Inkscape) Transparency is used (non-zero) for the text in Inkscape, but the package 'transparent.sty' is not loaded}%
    \renewcommand\transparent[1]{}%
  }%
  \providecommand\rotatebox[2]{#2}%
  \newcommand*\fsize{\dimexpr\f@size pt\relax}%
  \newcommand*\lineheight[1]{\fontsize{\fsize}{#1\fsize}\selectfont}%
  \ifx\svgwidth\undefined%
    \setlength{\unitlength}{209.76377106bp}%
    \ifx\svgscale\undefined%
      \relax%
    \else%
      \setlength{\unitlength}{\unitlength * \real{\svgscale}}%
    \fi%
  \else%
    \setlength{\unitlength}{\svgwidth}%
  \fi%
  \global\let\svgwidth\undefined%
  \global\let\svgscale\undefined%
  \makeatother%
  \begin{picture}(1,0.40540539)%
    \lineheight{1}%
    \setlength\tabcolsep{0pt}%
    \put(0,0){\includegraphics[width=\unitlength,page=1]{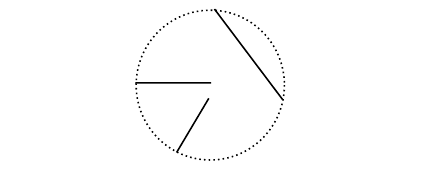}}%
    \put(0.47689826,0.2383612){\color[rgb]{0,0,0}\makebox(0,0)[lt]{\lineheight{1.25}\smash{\begin{tabular}[t]{l}$0$\end{tabular}}}}%
    \put(0.56691171,0.31019885){\color[rgb]{0,0,0}\makebox(0,0)[lt]{\lineheight{1.25}\smash{\begin{tabular}[t]{l}$V$\end{tabular}}}}%
  \end{picture}%
\endgroup%
}
		\caption{The black lines represents the discontinuity set $K$ of $u$. Corollary~\ref{limit tip} yields that the component $V$ is uniformly away from $0$.  }
		\label{fig:gamma}
	\end{figure}

We remark that, according to \cite[Theorem 6.1.1]{DF2023} and Corollary~\ref{limit tip}, the only missing part towards Mumford--Shah conjecture (see e.g. \cite[Conjecture 1.2.1]{DF2023}) is that, the set of tip points of $u$ is discrete in $\Omega$. 

In the last section, we prove a Morrey-type estimate in John domains, Lemma~\ref{growth of u}, and specifically, apply a version of \cite[Theorem 1.7]{SZ2024} to cover a neighborhood of the origin with uniformly finitely many John subdomains. These results finally give us the desired estimate. 

\

{\bf Acknowledgement}: The author would like to sincerely thank Prof. A. Figalli for the insightful discussions on this topic, and also expresses his gratitude to Prof. C. De Lellis for pointing out that the current result is a local version of the original  estimate for global minimizer proved by David and L\'eger \cite{DL2002}. 

\section{Preliminaries and Proof of Proposition~\ref{dichotomy}}\label{Preliminaries}
 Let us fix some notation. We denote the $k$-dimensional Hausdorff measure by $\mathcal H^{k}$.  For a given open set $\Omega\subset\mathbb R^n$, we denote by  $BV(\Omega)$ the space of functions of bounded variation in $\Omega$,  
whose weak (or distributional) gradient $Du$ is a vector-valued Radon measure. We write $D^a u$ for the absolutely continuous part of $Du$ and $D^s u$ for its singular part. 
The set of approximately continuous points of $u$ is denoted by $\mathcal C_{u}$, each of which is a Lebesgue point of $u$. We write $S_u=\Omega\setminus \mathcal C_u$ the Borel set of approximate discontinuity of $u$. 
Furthermore, for $\mathcal H^{n-1}$-almost every $x\in S_u$, there exists a direction $\nu_u\in \mathbb S^{n-1}$ and two  numbers $u_\pm\in \mathbb R$ so that $u_-(x) <u_+(x)$ and
$$\lim_{r\to 0} \int_{B_r^+(x,\,\nu_u(x))} |u(y)-u_+(x)|\, dy=0,$$
together with
$$\lim_{r\to 0} \int_{B_r^-(x,\,\nu_u(x))} |u(y)-u_-(x)|\, dy=0,$$
where 
$$B_r^+(x,\,\nu_u(x))=\left\{ y\in B_r(x)\colon (y-x)\cdot \nu_u(x) >0\right\},$$
and $B_r^-$ is defined similarly. These points are called \emph{jump points} of $u$. When 
$$D^s u =(u_+- u_-)\nu_u \mathcal H^{n-1}| _{S_u},$$
then $u\in SBV(\Omega), $ the space of \emph{special functions of bounded variation}. 
When $u\in SBV(\Omega)$ is a local minimizer, we usually denote by $K$ the closure of its jump set $\overline {S}_u$  unless explicitly stated otherwise. We say $x$ is a \emph{tip point} of $u$ if $x\in \overline{S}_u\setminus S_u$, which is a Lebesgue point of $u$, particularly.

For ease of reference, we assume $\lambda=1$ in the Mumford--Shah energy throughout the paper, unless explicitly stated otherwise. Moreover, for the ease of readablity, we suppress the subindex and write $MS:=MS_1$. 
Suppose $u\in SBV(B_r)$ is a local $1$-minimizer with $0\in \overline{S}_u$, then it follows from \cite[Remark 7.13]{AFP2000} that
$$
u_r(x)=r^{-\frac {n-1} n}u(rx)\in SBV(B_1)
$$
is also a local $1$-minimizer in $B_{1}$. Moreover, 
$$MS(u,\,B_r)=r^{n-1}MS( u_r,\,B_1).$$

 For a (rectifiable) curve $\gamma$, we denote by $\ell(\gamma)$ the Euclidean length of $\gamma$. When $\gamma$ is an arc (i.e. an injective curve), for any pair of points $x,\,y\in \gamma$, denote by $\gamma[x,\,y]$ a subarc joining them.  For a measurable set $A\subset \mathbb R^2$, we write
$$\bint_{A} u\, dx: =\frac 1 {|A|}\int_{A} u\, dx. $$

Let us begin with the following result, called the Ahlfors regularity of the local minimizer, which holds for every point in $K$. It is clear that the density 
$$r^{-1}  \mathcal{H} ^{1}\left( K \cap B_r \right)$$
is invariant under scaling. From this point forward, we focus exclusively on the planar case, specifically when $n=2$. 

\begin{lem}[{\cite[Theorem 2.6]{F2003}, \cite[Corollary 3.3]{BL2014}}]\label{A reg} 
 Suppose that $u\in SBV(\Omega)$ is a local minimizer and $0\in K$. 
There exists a constant $r_1>0$ so that for each $0<r<r_1$ and $B_r\subset \Omega$, we have
$$\frac 1 {C}\le r^{-1}  \mathcal{H} ^{1}\left( K \cap B_r \right) \le C, $$
and
$$ \int_{B_r} |Du|^2 \,dx \le Cr. $$
\end{lem}
 
We also record the following $\ez$-regularity of David, which provides a criterion to individuate tip points.

\begin{lem}[{\cite[Proposition 60.1]{D2006}}] \label{eps regularity}
 Suppose that $u$ is a minimizer and $0\in K$. There exists $\ez_0>0$ and  $\eta>0$ so that, whenever $x\in K$ satisfies
$$\int_{B_r(x)} |Du|^2\, dx\le \ez_0 r,$$
then $x$ is a jump point and $K\cap B_{\eta {r}(x)}$ is a $C^1$-curve or a $C^1$-spider (see \cite[Section 51 \& 53]{D2006} for the definitions). 
\end{lem}

Now we are ready to show Proposition~\ref{dichotomy}.

\begin{proof}[Proof of Proposition~\ref{dichotomy}]
Up to a translation, we may assume that $x$ is the origin. 
If \eqref{poincare} holds, then  for some $c_0>0$ and any sequence $r_k\to 0$ we have
$$
\limsup_{k\to \infty}   {r_k^{-1}\int_{B_{r_k}} |Du|^2\, dx}\le c_0 \limsup_{k\to \infty} \left(\inf_{c\in \mathbb R} \bint_{  B_{r_k}}|u-c|^2 \,dx\right). 
$$
Suppose that, on the contrary,  $0$ is a tip point, which is particularly a Lebesgue point of $u$. Then the right-hand side of the inequality above goes to $0$ as $k\to \infty$. 
This yields that 
$${r^{-1}\int_{B_{r}} |Du|^2\, dx}$$
is small whenever $r>0$ is sufficiently small. However, Lemma~\ref{eps regularity} implies that $0$ cannot be a tip point. This leads to a contradiction and thus $0$ must be a jump point. 

Now suppose that \eqref{poincare} fails, then by Lemma~\ref{A reg}, there exists a sequence $r_k\to 0$ so that for any  $M>0$ 
$$\lim_{k\to \infty} \left(\inf_{c\in \mathbb R} \bint_{  B_{r_k}}|u-c|^2 \,dx\right) \le  M^{-1}\lim_{k\to \infty} r_k^{-1} {\int_{B_{r_k}} |Du|^2\, dx}\le C M^{-1}.$$ 
Thus the origin is a Lebesgue point of $u$ as $M\to \infty$ (and $r_k\to 0$), which yields that $0$ is a tip point.  
\end{proof}

Let us also recall the definition of John domain. 
\begin{defn}
    For   $J\ge 1$, a (bounded) domain $\Omega\subset \mathbb R^n$ is said to be  $J$-John provided that, there exists a distinguished point $x_0\in \Omega$ so that, for every $x\in \Omega$, there exists an arc $\gamma\subset \Omega$ starting from $x$, ending at $x_0$ and satisfying the following condition: 
\begin{equation}\label{John curve}
\ell(\gamma[x,\,y])\le J\dist(y,\,\partial \Omega) \quad \text{ for any $y\in \gamma$}, 
\end{equation}
where $\ell(\gamma[x,\,y])$ denotes the length of the subcurve of $\gamma$ joining $x$ to $y$. 
We usually call $x_0$ the  {John center} of $\Omega$ and $\gamma$ the  {John curve} joining $x_0$ and $x$. 

Furthermroe, let $\gz\subset  \Omega$ be a curve joining $x$ to  $x_0$, and define the $J$-$carrot$ with the vertex $x$ and the core $\gamma$ joining $x$ to $x_0$ as
		$$car(\gz, J):=\bigcup\big\{B(y,\ell(\gamma[x,y])/J): y \in \gamma \setminus\{x\}\big\}.$$
Then  a (bounded) domain  $\Omega\subset {\mathbb R}^n$ is $J$-John  with the center $x_0 \in \Omega\cup \{\infty\}$, if for each  point  $x\in \Omega$, there exists a curve $\beta\subset \Omega$ joining $x$ to  $x_0$ so that $car(\beta, J)\subset \Omega$.
\end{defn}

We record the following result, which says that $\Omega\setminus K$ is locally John. 

\begin{thm}{\cite[Proposition 68.16]{D2006}}\label{local john}
Let $u\in SBV(\Omega)$ be a local minimizer  and $x\in \Omega\setminus K$.
Then there exists an absolute constant $J\ge 1$ so that, for any 
$$0<r\le \frac 1 2 \dist(x,\,\partial \Omega)=:r_3,$$ 
one can find an arc  $\gamma\subset \Omega\setminus K$ starting from $x$, escaping $B_r(x)$ (i.e. for a parametrization $\gamma\colon [0,\,1]\to \Omega$ with $\gamma(0)=0$, there exists $t\in (0,\,1)$ that $\gamma((t,\,1])\cap B_r(x)=\emptyset$), and satisfying that, for any $y\in \gamma,$
$$\ell(\gamma[x,\,y])\le  J \dist(y,\,K). $$
In particular, there exists a ball with radius $J^{-1} r$ contained in   $B_r(x)\cap (\Omega\setminus K)$.
\end{thm}

\section{Proof of Theorem~\ref{main thm}}

Let $0\in \Omega$ be a tip point of $u$, and assume that $u(0)=0$ up to an additive constant. 
Recall that tip points are Lebesgue points of $u$, i.e. \
$$\lim_{r\to 0}\bint_{B_r} |u|\, dx= 0.$$

Write $K=\overline{S}_u$. 
Recall the definition of $J$-carrot.  Then Theorem~\ref{local john} implies that, for every $x\in \Omega\setminus K$, one can find a $J$-carrot contained in $\Omega\setminus K$ with core $\gamma$ and vertex $x$, which escapes $B_r(x)$. 
		
The carrot condition above is equivalent to the following $M$-Boman chain condition, quantitatively: There exists $M\ge 1$ and a sequence of balls $\{U_i\}_{i=0}^\infty$ converging to $x$ so that, 
\begin{itemize}
\item $U_0=B_{\frac{\ell(\gamma)}{4J}}(x_0)$, 
\item $M^{-1}\diam(U_{i+1})\le \diam(U_i)\le M \diam(U_{i+1}),$
\item There exists $R_i\subset U_i\cap U_{i+1}$ so that 
$U_i\cup U_{i+1}\subset MR_i,$
\item $\sum_{i} \chi_{U_i}\le M.$
\end{itemize}
see e.g. \cite[Theorem 9.3]{HK2000}. 
We record the following results for John domains.

\begin{lem}\label{growth of u}
Let  $C_0\ge 4,\,r>0$ and $U\subset\mathbb R^2$ be a (bounded) $J'$-John  domain with center $x_0\in \partial B_{3r}\cap U$ such that $0\in U\subset B_{C_0r}$, $J'\ge 1$. Moreover, assume that there exists $C_1>0$ so that $u\in W^{1,\,2}(U)$ satisfies
\begin{equation}\label{Du upper bound}
    \int_{B_s(z)\cap U} |Du|^2\, dx \le C_1 s
\end{equation} 
for any $z\in U$ and $0<s<r$. Then 
\begin{equation*}
|u(x)-u(y)|\le C(C_0,\,C_1,\,J') r^{\frac 1 2} \quad \text{ for almost every pair of points } x,\,y\in U\subset B_{5r}.
\end{equation*}
\end{lem}
\begin{proof}
Fix $x,\,y\in U$ which are Lebegue points of $u$. 
Since carrot and Boman chain conditions are equivalent, we can find two sequences of balls $\{U_i\}_{i=0}^{+\infty}$ and $\{U_{j}\}_{j=0}^{-\infty}$ of $M$-Boman chains from $x$ and $y$ to $x_0$, respectively, where
$$M=M(J'), \quad U_0=B_{c(C_0,\,J')r}(x_0),\quad U_i\to x  \ \text{ as } \ i=\to +\infty,\quad \,U_j\to y \ \text{ as } \ j\to -\infty.$$
Then by writing in telescoping sum and applying Poincar\'e inequality on $u$, we have
\begin{align*}
|u(x)-u(y)|\le & \sum_{k=-\infty}^{+\infty} \left|\bint_{U_k} u \,dx - \bint_{U_{k+1}} u \,dx \right| \\
\le & \sum_{k=-\infty}^{+\infty} \bint_{U_k} \left|u  - \bint_{U_{k+1}} u\right| \,dx \\
\le &  C(J')  \sum_{k=-\infty}^{+\infty} \diam(U_k)  \left( \bint_{U_{k}\cup U_{k+1}}|Du|^2\, dx\right)^{\frac 1 2}.
\end{align*}
As the assumption \eqref{Du upper bound} gives
$$  \bint_{U_{k}\cup U_{k+1}}|Du|^2\, dx \le C(C_1,\,J) \diam(U_k)^{-1}, $$
it follows that
\begin{align*}
\sum_{k=-\infty}^{+\infty} \diam(U_k)  \left( \bint_{U_{k}\cup U_{k+1}}|Du|^2\, dx\right)^{\frac 1 2}\le   & C(C_1,\,J') \sum_{k=-\infty}^{+\infty} \diam(U_k) ^{\frac 1 2} \le C(C_0,\,C_1,\,J')r^{\frac 1 2},
\end{align*}
where we applied the fact that both $\{\diam (U_i)\}_{i=0}^{+\infty}$ and $\{\diam (U_j)\}_{j=0}^{-\infty}$  are geometric series in the last inequality. Then our lemma follows from the chain of inequalities.  
\end{proof}

The following lemma is a local version of \cite[Theorem 1.7]{SZ2024} and we record its proof in Appendix~\ref{a1}. 

\begin{lem}\label{local john decompose}
Suppose that $0\in \overline {S}_u\subset \mathbb R^2$ is a tip point. 
Then for  $r_3>0$ defined in Theorem~\ref{local john}  and  any $0< r<r_3/9 $, 
there exist $C=C(J)>0$, $N=N(J)\in \mathbb N$ and $J'=J'(J)>0$ so that, we can cover $B_r\setminus \overline{S}_u$ by (the closure of) at most $N$-finitely many $J'$-John domains $W_{j,\,r}\subset B_{Cr}\setminus  \overline{S}_u$, $1\le j\le N$.

In particular, there exists a constant $C_2=C_2(J)>0$ such that, for every point $x\in W_{j,\,r}$, one can find a rectifiable curve $\beta_x\subset W_{j,\,r}$, which joins $x$ to a point $w_{j,\,r}\in \partial B_{3r}\cap W_{j,\,r}$,  as the core of a $J'$-carrot satisfying
$$\ell(\beta_x)\le C_2 r. $$
\end{lem}

This lemma together with Lemma~\ref{growth of u} and Lemma~\ref{local john decompose} implies the following  theorem. 

\begin{proof}[Proof of Theorem~\ref{main thm}]
Observe that Theorem~\ref{local john} together with Lemma~\ref{local john decompose} implies that, for any $0<r<r_3/9$, we can cover $B_r\setminus \overline{S}_u$ by at most $N$-finitely many (nonempty) $J'$-John domains $W_{j,\,r}$. 

Now by recalling that $u(0)=0$ and the $L^2$-estimate on the gradient from Lemma~\ref{A reg}, we employ  Lemma~\ref{growth of u} to $u$ on each of the John domains $W_{j,\,r}$ to conclude that,  for $0<r<r_0:= \min\{r_1,\,r_2,\,r_3/9\}$, 
$$|u(x)-u(0)|=|u(x)|\le  C r^{\frac 1 2} \quad \text{ for any } \  x\in W_{j,\,r}.$$
Moreover, Corollary~\ref{limit tip} yields that, for 
any $0<r<r_0\le r_2$, one has
$0\in \partial W_{j,\,r}$ for every  $W_{j,\,r}$. This yields our desired estimate as the number of John domains is at most $N$. 
\end{proof}

\appendix

\section{Proof of Lemma~\ref{local john decompose}}\label{a1}

To start with, we record the following proposition. While this technique has been commonly employed in previous manuscripts, such as \cite{NV1991}, it has not been explicitly formulated, to the best of our knowledge, in the context of our present work.

	\begin{prop}\label{length car pro}
		Let $J\ge 1$. Assume that $\gamma \subset\mathbb{R}^2$ is a locally rectifiable curve joining $x$ to $y$, where $x,\,y\in \mathbb{R}^2$. Then
		$ car(\gamma,J)$ is a $J$-carrot John domain. 
		
		To be more specific,  for any $z\in car(\gamma,J)$, we can find a rectifiable curve $\gamma_z$ joining $z$ to $y$, such that for some $\eta\in \gamma$, we have
		$$\gamma[\eta,y]=\gamma_z[\eta,y]$$
		and for each $a\in \gamma[\eta,y]$,  
		\begin{equation}\label{length}
			\ell(\gamma_z[z,a])\le  \ell(\gamma[x,a]),\quad car(\gamma_z,J)\subset car(\gamma,J).
		\end{equation}
	\end{prop}
	\begin{proof}
		For any $z\in car (\gamma,J)$, the definition of $car (\gamma,J)$ yields a ball 
		$$B(\eta,\ell(\gamma[x,\eta])/J)\subset car (\gamma,J)$$
		for some points $\eta\in \gamma\setminus \{x\}$ so that $z\in B(\eta,\ell(\gamma[x,\eta])/J)$.
		
		Let $L_{z,\eta}$ be the line segment joining $z$ to $\eta$ and then $\gamma_z:=L_{z,\eta}\cup \gamma[\eta,y]$ is locally  rectifiable curve joining $z$ to $y$. When $a\in L_{z,\eta}$,  
		\begin{equation}\label{a41}
			\ell(\gamma_z[z,a])\le d\left(a,\partial B(\eta,\ell(\gamma[x,\eta])/J)\right)\le \ell(\gamma[x,\eta])/J.
		\end{equation}
		When $a\in \gamma[\eta,y]$, by applying \eqref{a41} with $a=\eta$ there, 		we have 
		\begin{align*} 
			\ell(\gamma_z[z,a])&\le \ell(\gamma_z[z,\eta])+\ell(\gamma_z[\eta,a])\le \frac{\ell(\gamma[x,\eta])}{J}+\ell(\gamma[\eta,a])\nonumber\\
			& \le \ell(\gamma[x,\eta])+\ell(\gamma[\eta,a])=\ell(\gamma[x,a]).
		\end{align*}
		To conclude, we obtain that
		$$\ell(\gamma_z[z,a])\le  \ell(\gamma[x,a]),$$
		which is the first formula of \eqref{length}. 
		The second one   follows directly from out construction of $car(\gamma_z,J)$ and $car(\gamma,J)$, and we conclude the lemma. 
	\end{proof}
	
	\subsection{A decomposition $V_{j,\,r}$ of $B_r\setminus K$}
	Recall that $K=\overline S_{u}$. 
	Now  for any $x\in B_r \setminus K$ with $B_r\subset \Omega, 0<r<r_0/9$, 
	we choose an escaping  (John) curve $\gamma_x\subset \Omega\setminus K$ from $x$  with $car(\gamma,J)\subset B_r\setminus K$.
	Although there could be many choices of curves for $x\in\Omega\setminus K$, we just  choose one of them. Let $\Gamma=\{\gamma_x\}_{x\in B_r\setminus K }$ be the collection of these chosen curves. In what follows, for any points $x\in B_r \setminus K$, $\gamma_x$ always refers to this particular choice of escaping curve.
	
	Note that for  $0<r<r_0/9$, we have  $B_{r}\cap K \neq \emptyset$ as $0\in K$. Our first step is to decompose $B_r\setminus K$ into finitely many subsets $V_{j,\,r}$ so that, there exists a collection $\mathcal{B}_{j\,,r}$ of at most $C(J)$-many balls, whose center is on  $\partial B_{3r}$ and whose radius is at least $J^{-1}r$, satisfying that, 
	for  any $x\in V_{j,\,r}$, we can find a ball $B\in\mathcal{B}_{j,\,r}$ with
	$$\gamma_x\cap B\neq\emptyset.$$

	To this end, observe that, according to Theorem~\ref{local john},  for each $x\in B_r\setminus K$ and $\gamma_x\in \Gamma$, there exists a point
	\begin{equation}\label{xR def}
		x_r\in\gamma_x\cap\partial B_{3r}  
	\end{equation}
	so that 
	\begin{equation}\label{lower size}
		2r\le\ell(\gamma[x,x_r])\le J\dist(x_r,K).
	\end{equation}
	
	Consider the collection of closed balls 
	\begin{equation}\label{ball collect}
		\{\overline{B}_x\}_{x\in B_r\setminus K}:=\left\{\overline{B}\left(x_r,\frac{\dist(x_r,K)}{2}\right)\right\}_{x\in B_r\setminus K}.
	\end{equation}
	Then thanks to \eqref{lower size} and $0\in K$, we  obtain that 
	\begin{equation}\label{radius larg}
		\frac{r}{J}\le \frac{\dist(x_r,K)}{2}\le \frac 3 2 r,
	\end{equation}
	and hence  $B_x\cap B_r=\emptyset$. 
	
	We next let 
	$$A_r:=\bigcup_{x\in \overline B_{r}\setminus K}\{x_r\}$$
	be the collection of the centers of $B_x$'s. By Bescovitch's covering theorem, there exists a subcollection  $\{\overline{B}_i\}_{i\in\mathbb{N}}$ of $\{\overline{B}_x\}_{x\in B_r\setminus K}$ consisting of at most countably many balls, such that 
	
	\begin{equation}\label{bescovitch}
	    \chi_{A_r}(z)\le \sum_{B_i} \chi_{\overline{B}_i}(z)\le C, \qquad  \forall z\in B_{5r}\setminus K. 
	\end{equation}

	
	Recall that by \eqref{radius larg}
	$$B_i\subset B_{5r}\setminus \overline{B}_{r}$$
	and
	$|B_{i}|\ge c(J)r^n. $
	Thus we have at most $C(J)$-many elements in $\{\overline{B_i}\}$ by \eqref{bescovitch}. 
	As a result, the union of balls 
	$$\bigcup_{i}\overline{B}_i$$
	has at most $N= {N}(J)$ components $U_{j,r}$
	for $j\in\{1,\cdots,N_r\}$ and 
	$N_r\le  N.$
	By defining $U_{j,\,r}$ to be empty for $j> N_r$, we may assume that there exists exactly $ N$ components $U_{j,r}$, and each $U_{j,\,r}$ contains at most $\hat N=\hat N(J)$ balls. 
	We write $\mathcal {B}_{j,r}$ as the collection of balls $B_i$ contained in each component $U_{j,r}$.
	
	Now it follows from our construction, for any $x\in B_r\setminus K$, there exists some $1\le j\le N_r$ so that, $x_r\in \gamma_x$ is covered by a ball in $\mathcal{B}_{j,r}$. Thus, by  defining 
	\begin{equation}\label{VR}
		V_{j,r}:=\{x\in B_r\setminus K:x_r\in D \text{ for some }D\in\mathcal{B}_{j,r}\},
	\end{equation}
	we obtain the desired decomposition of $B_r\setminus K$. The set $V_{j,\,r}$ is defined to be empty if $U_{j,\,r}$ is empty.

The following lemma is a version of {\cite[Proposition 3.2]{SZ2024}}.

	\begin{prop}\label{omega cons}
		For  $1\le j\le N$ with $N= N(J)$ defined above, the set $W_{j,\,r}$ is either empty (if $V_{j,\,r}$ is empty), or for any fixed $y\in V_{j,\,r}$ together with the escaping point $y_r\in \partial B_{3r}$, the set
		\begin{equation}\label{Omega def}
			W_{j,\,r}:=car(\gamma_y[y,y_r],J)\cup \bigcup_{x\in V_{j,r}}car(\beta_x,J').
		\end{equation}
		is a $J'$-John domain with John center $y_r$, where $J'=J'( J)>0$
 and  $\beta_x$ is a rectifiable curve joining $x$ to $y_r$ satisfying $\gamma_{x}[x,x_r]\subset \beta_x$; recall that $\gamma_x$ is the escaping curve starting from $x$. 
		 
		 Moreover, there exists a constant $C_3=C_3(J)\ge 4$ so that, the curve $\beta_x$  joining $x\in W_{j,r}$
		  to $y_r$ is  the core of a $J'$-carrot satisfying 
		$$\ell(\beta_x)\le C_3 r$$
		and
		\begin{equation}\label{close omega}
			V_{j,r}\subset W_{j,r},\quad car(\beta_x,J')\subset W_{j,r}\subset (\Omega\setminus K)\cap B_{2 C_3 r}.
		\end{equation}
	\end{prop}
	\begin{proof}
		Suppose that $ V_{j,r}$ is non-empty, and fix a point $y\in V_{j,r}$. Then the corresponding escaping point $y_r\in \gamma_y\cap \partial B_{3r}$ is covered by some ball $D_1\in \mathcal{B}_{j,r}$ according to \eqref{bescovitch}.  Then we can join $y_r$  to the center $\hat{x}_1$ of $D_1$ by the line segment $L_{y_r,\hat{x}_1}\subset D_1$. 
		
		Now for any $x\in V_{j,r}$, we claim that there exists  a rectifiable curve $\beta_x\subset \Omega\setminus K$ as the core of a $J'$-carrot joining $x$ to $y_r$, such that $\gamma_{x}[x,x_r]\subset \beta_x$ and
		\begin{equation}\label{base point ball}
			car(\beta_{x},J')\subset \Omega\setminus K.
		\end{equation}
		
		Indeed,  the escaping point $x_r\in  \partial B_{3r}\cap D_2$ is also covered by another ball $D_2\in\mathcal{B}_{j,r}$ as $x\in V_{j,r}$. Moreover, we can also joint $x_r$  to the center $\hat{x}_2$ of $D_2$ by the line segment $L_{x_r,\hat{x}_2}\subset D_2$.

		Recall that $U_{j,r}$ is connected and consists of at most $\hat{N}$-many  balls from $\mathcal{B}_{j,r}$.  
		This implies that $\hat{x}_1$ and $\hat{x}_2$ can be joined by a union of at most  $\hat {N}$-many line segments  with the endpoints being the centers of  balls in  $\mathcal{B}_{j,r}$. Therefore, combining with $L_{y_r,\hat{x}_1}$ and $L_{x_r,\hat{x}_2}$, we can join $x_r$ to $y_r$ by a polyline $\gamma_{x_r,y_r}$.  
		
		We show that 
		$$\beta_x:=\gamma_x[x,\,x_r]\cup \gamma_{x_r,\,y_r}$$
		is the desired John curve. 
		To this end,  we estimate the length of $\beta_x$ and the distance $\dist(\eta,K)$ for any $\eta\in \beta_x$, respectively. 
        
        We start with the estimate on the length of $\beta_x$. Thanks to \eqref{radius larg}, for any pair of intersecting balls $D,D'\in \mathcal{B}_{j,r}$,  the line segments $L$ joining the center of $D$ with radius $s$ to the center of $D'$ with radius $s'$ satisfies 
		\begin{equation}\label{length estimate 2}
			L\subset D\cup D'\quad \text{and}\quad \ell(L)\le s+s'\le 4r.
		\end{equation}
		In particular, \eqref{radius larg} together with the facts that $L_{x_r,\hat{x}_2}\subset D_2$ and that $L_{y_r,\hat{x}_1}\subset D_1$ also yields
		$\ell(L_{x_r,\hat{x}_2})\le 2r$, $\ell(L_{y_r,\hat{x}_1})\le 2r$. 
		Therefore employing   \eqref{length estimate 2} and \eqref{lower size},  the construction of $\beta_x$ tells
		\begin{align}\label{length beta}
			\ell(\beta_x)&\le \ell(\gamma_x[x,\,x_r])+\ell(\gamma_{x_r,\,y_r}) \nonumber\\
			&\le J\dist(x_r,\,K)+\ell(L_{x_r,\hat{x}_2})+\ell(L_{y_r,\hat{x}_1})+4 \hat N r \nonumber\\
			&\le C(J) r=:C_3r;
		\end{align}
		we may assume that	$C_3\ge 4$. 
		This gives the first part of the proposition.
		
		Towards \eqref{close omega}, for any $\eta\in \beta_x$ we need to estimate the distance $\dist(\eta,K)$ from above.  First of all,  note that  when $\eta\in\gamma_{x_r,y_r}$, there exists  some ball $D_{\eta}\in \mathcal{B}_{j,r}$ containing $\eta$. Then combining \eqref{lower size}, \eqref{ball collect} and \eqref{radius larg}, we get 
		\begin{equation}\label{small ball}
			\dist(\eta,K)\ge \dist(D_\eta,\,K)\ge  \frac{r}{J}.
		\end{equation}
		Let 
		\begin{equation}\label{carrot construct 2}
			J':=C_3 J.
		\end{equation}
		Then  combining \eqref{lower size}, \eqref{length beta} and \eqref{small ball},   we conclude
		$$\ell(\beta_x[x,\eta])\le \ell(\beta_x)\le C_3 r\le  J'\dist(\eta,K) \quad \text{ when } \eta\in\gamma_{x_r,\,y_r}. $$
        
		On the other hand,  when  $\eta\in \gamma_x[x,\,x_r]$, since our construction yields $\beta_x[x,\,\eta]=\gz_x[x,\,\eta]$, which is particularly contained in a John curve, it follows that
		$$ \ell(\beta_x[x,\eta])\le J \dist(\eta,\,K)\le J'\dist(\eta,K).$$
		This   implies \eqref{base point ball}.
		Moreover by Proposition~\ref{length car pro}, 
		every point $w\in car(\beta_x,J')$ also can be joined to $y_r$ by a rectifiable curve $\hat{\gamma}_w$ satisfying 	$$\ell(\hat{\gamma}_w)\le \ell(\beta_x)\quad \text{and }\quad car(\hat{\gamma}_w,J')\subset car(\beta_x,J').$$
		Hence, by employing  \eqref{base point ball}, the arbitrariness of $x$ gives the second formula in  \eqref{close omega}. The first formula in \eqref{close omega} holds due to $x\in Cl(car(\beta_x,J'))$, the closure of the carrot,  for any $x\in V_{j,r}$.
	\end{proof}

Now Lemma~\ref{local john decompose} follows immediately from  Proposition~\ref{omega cons}, where $w_{j,\,r}$ is chosen to be $y_r$ in the Proposition~\ref{omega cons}.

\end{document}